\documentclass[reqno,centertags,11 pt, a4]{amsart}
\usepackage{a4wide,amsmath,amsthm,amscd}
\usepackage{amssymb,latexsym}
\usepackage[dvips]{epsfig}
\usepackage{array}
\newcommand{\F}{\mathbb F}

\newcommand{\Fpbar}{\overline{\F}_p}
\newcommand{\R}{\mathbb R}
\newcommand{\C}{\mathbb C} 
\newcommand{\Q}{\mathbb Q}
\newcommand{\Z}{\mathbb Z}  

\newcommand{\A}{\mathbb A} 
\newcommand{\T}{\mathbf T}
\newcommand{\m}{\mathfrak m}

\newcommand{\vphi}{\varphi}

\newcommand{\Frob}{\mathrm{Frob}}

\newcommand{\ind}{\mathrm{ind}}

\newcommand{\Gal}{\mathrm{Gal}}

\newcommand{\Sym}{\mathrm{Sym}}

\newcommand{\tr}{\mathrm{tr}}

\newcommand{\Spec}{\mathrm{Spec\mbox{ } }}
\newcommand{\GL}{\mathrm{GL}}
\newcommand{\p}{\mathfrak{p}}

\newcommand{\tensor}{\otimes}
\newtheorem{theorem}{Theorem}[section]
\newtheorem{proposition}[theorem]{Proposition}
\newtheorem{lemma}[theorem]{Lemma}
\newtheorem{dream}[theorem]{Dream}

\theoremstyle{definition}

\newtheorem{conjecture}[theorem]{Conjecture}
\newtheorem{definition}[theorem]{Definition}

\theoremstyle{remark}
\newtheorem{remark}[theorem]{Remark}

\begin{document} 
\title[Serre's modularity conjecture]{Serre's Modularity Conjecture}
\author[Michael M. Schein]{Michael M. Schein}
\address{Department of Mathematics, Bar-Ilan University, Ramat Gan 52900, Israel}
\email{mschein@math.biu.ac.il}
\thanks{The author was supported by grants from the German-Israeli Foundation for Scientific Research and Development and from the Pollack Foundation.}

\maketitle

\section{Introduction}

These notes are based on lectures given by the author at the winter school on Galois theory held at the University of Luxembourg in February 2012.  Their aim is to give an overview of Serre's modularity conjecture and of its proof by Khare, Wintenberger, and Kisin \cite{KWfirst} \cite{KWsecond} \cite{Kisin}, as well as of the results of other mathematicians that played an important role in the proof.  Along the way we will remark on some recent work concerning generalizations of the conjecture.  We have tried as much as possible to concentrate on giving a broad picture of the structure of the arguments and have ignored technical details in places.  Some results are given incomplete statements, where we have chosen not to list technical hypotheses; we request the reader's forbearance.

Let $F$ be a totally real number field.  We will denote by $G_F$ the absolute Galois group $\Gal(\overline{\Q}/F)$.  It was shown in Prof. B\"{o}ckle's lectures in this volume how,  under some hypotheses, a Hilbert modular eigenform $f$ over $F$ gives rise to a compatible system $\{ \rho_{f, v} \}$ of $p$-adic Galois representations; see \cite{Jarvis} for the most general theorem.  These representations are extracted from the cohomology of a suitable algebraic variety, and this construction is more or less the only method we have for obtaining $p$-adic Galois representations.  Therefore the following question is of acute interest: given a Galois representation $\rho: G_F \to \GL_2(\overline{\Q}_p)$, when is $\rho$ {\emph{modular}}, i.e. when does there exist a Hilbert modular eigenform $f$ and a place $v | p$ of $F$ such that $\rho \simeq \rho_{f, v}$?

This question is a very difficult one.  We will split it into two questions, which are still very difficult, by introducing the notion of the reduction of a Galois representation.  The following result is classical; the proof given here is attributed to N. Katz and appears in print, for instance, at the beginning of section 2 of \cite{Skinner}.

\begin{proposition} \label{finiteext}
Let $G$ be a compact Hausdorff group, and let $\rho: G \to \GL_n(\overline{\Q}_p)$ be a continuous representation.  Then $\rho$ is equivalent to a representation $\rho^\prime$ such that $\rho^\prime(G) \subset \GL_n(\mathcal{O}_L)$, where $\mathcal{O}_L$ is the ring of integers of some finite extension $L / \Q_p$.
\end{proposition} 
\begin{proof}
Since $G$ is compact and Hausdorff, it admits a Haar measure $\mu$; without loss of generality, $\mu(G) = 1$.  Now,
$$ \GL_n(\overline{\Q}_p) = \bigcup_{[L: \Q_p] < \infty} \GL_n(L) $$
and hence
$$ G = \bigcup_{[L : \Q_p] < \infty} \rho^{-1}(\GL_n(L)).$$

Since there are countably many finite extensions $L/ \Q_p$, there must be some $L$ such that $\mu(\rho^{-1}(\GL_n(L))) > 0$.  Hence, $\rho^{-1}(\GL_n(L)) \subset G$ is a closed subgroup of finite index.  Then $\rho^{-1}(\GL_n(\mathcal{O}_L))$ is an open subgroup of the compact group $\rho^{-1}(\GL_n(L))$, so it has finite index in it and thus in $G$.  Let $g_1, \dots, g_m$ be a collection of coset representatives for $\rho^{-1}(\GL_n(\mathcal{O}_L))$.  Let $\Lambda \subset L^n$ be the lattice generated by $\rho(g_1) \mathcal{O}_L^n, \dots, \rho(g_m) \mathcal{O}_L^n$.  This is a lattice of maximal rank, so $\Lambda \simeq \mathcal{O}_L^n$.  Furthermore, $\Lambda$ is stable under the action of $G$.  Let $T \in \GL_n(L)$ be a linear transformation that takes $\mathcal{O}_L^n$ to $\Lambda$, and set $\rho^\prime(g) = T^{-1} \rho(g) T$.  
\end{proof}

The proposition above applies, in particular, to continuous representations $\rho: G_F \to \GL_2(\overline{\Q}_p)$ such as we have been considering.  If $\rho$ is equivalent to $\rho^\prime: G_F \to \GL_2(\mathcal{O}_L)$, then we define the reduction $\overline{\rho}$ to be the semisimplification of the composition $\tilde{\rho}^\prime: G_F \overset{\rho^\prime}{\to} \GL_2(\mathcal{O}_L) \to \GL_2(k_L) \to \GL_2(\Fpbar)$, where $k_L$ is the residue field of $L$ and the inclusion $k_L \hookrightarrow \Fpbar$ is induced from $L \hookrightarrow \overline{\Q}_p$.  In other words, $\overline{\rho}$ is the direct sum of the Jordan-H\"{o}lder constituents of $\tilde{\rho}^\prime$.  This definition is independent of all choices, and we call $\overline{\rho}$ the reduction modulo $p$ of $\rho$.

\begin{remark}
Throughout these notes, except for Section \ref{modpllc}, we will usually use $\overline{\rho}$ to denote a mod $p$ Galois representation.  The bar simply serves to emphasize that we are dealing with a mod $p$ representation.  It does not necessarily mean that we have any $p$-adic representation $\rho$ in mind, of which $\overline{\rho}$ is to be the reduction.
\end{remark}

We say that a mod $p$ Galois representation $\overline{\rho} : G_F \to \GL_2(\Fpbar)$ is {\emph{modular}} if it ``arises from geometry'' in a way that will be made precise in the next section.  If $F = \Q$, then $\overline{\rho}$ is modular if and only if there exists a modular eigenform $f$ such that $\overline{\rho} \simeq \overline{\rho_{f, p}}$, while for larger totally real fields we will require a somewhat more subtle notion of modularity.  If a $p$-adic Galois representation $\rho$ is modular, then its reduction modulo $p$ will be modular as well.  We will consider two questions:

\begin{enumerate}
\item Let $\rho: G_F \to \GL_2(\overline{\Q}_p)$ be a $p$-adic Galois representation.  Suppose that $\overline{\rho}$ is modular.  Is $\rho$ modular?
\item Let $\overline{\rho}: G_F \to \GL_2(\Fpbar)$ be a mod $p$ Galois representation.  When is $\overline{\rho}$ modular?
\end{enumerate}

It is clear that if we knew complete answers to both of these questions, their union would resolve the question of when a general $p$-adic representation is modular.  Affirmative responses to the first question are known in a variety of different cases; results of this type are called modularity lifting theorems.  A conjectural response to the second question is given by Serre's modularity conjecture and its generalizations.  Serre's original conjecture, covering the case of $F = \Q$, is now a theorem of Khare, Wintenberger, and Kisin.  However, as we shall see, even if we do not know whether a mod $p$ Galois representation is modular, we can say a lot about the Hilbert modular forms $f$ that it could come from if it were modular.  The research towards resolving each of these two questions is tightly interconnected with work concerning the other, as shall become evident in these notes.

\subsection{Acknowledgements}
The author is very grateful to Sara Arias de Reyna, Lior Bary-Soroker, and Gabor Wiese, the organizers of the Winter School on Galois Theory held in Luxembourg in February 2012, for inviting him to present these lectures.  He is grateful to the audience for their stimulating questions, to Tommaso Centeleghe and Nicolas Billerey for allowing him to make use of the notes they took in his lectures, and to the anonymous referee for a thorough reading.  The exposition given here has in places drawn on expositions of similar material given elsewhere, such as the lectures by Toby Gee, Richard Taylor, and Teruyoshi Yoshida at an MSRI workshop on modularity in 2006, of Richard Taylor at the summer school on Serre's conjecture at Luminy in 2007, and those by Fred Diamond at the Galois Trimester at the Institut Henri Poincar\'{e} in Paris in 2010.  All errors and inaccuracies are, of course, entirely the responsibility of the author.

\section{Statement of Serre's modularity conjecture} \label{statement}
\subsection{The classical conjecture}
If we fix an embedding $\overline{\Q} \hookrightarrow \C$, then complex conjugation is a well-defined element $c \in G_{\Q}$.  Since it is an involution, any Galois representation $\overline{\rho}: G_{\Q} \to \GL_2(\Fpbar)$ must send $c$ to a matrix with determinant $\pm 1$.  We say that $\overline{\rho}$ is {\emph{odd}} if $\det \overline{\rho}(c) = -1$.  Similarly, if $F$ is a totally real field with $[F: \Q] = d$, then the $d$ embeddings $F \hookrightarrow \R$ induce $d$ complex conjugation automorphisms $c_1, \dots, c_d \in G_F$.  We say that $\overline{\rho}: G_F \to \GL_2(\Fpbar)$ is {\emph{totally odd}} if $\det \overline{\rho}(c_i) = -1$ for each $i = 1, \dots, d$.

The original statement of Serre's conjecture, which essentially dates back to the 1960's, appeared in some cases in \cite{SerreConjFirst}, and was properly published only in \cite{SerreConj}, is the following.

\begin{conjecture}
Let $\overline{\rho}: G_{\Q} \to \GL_2(\Fpbar)$ be a mod $p$ Galois representation.  If $\overline{\rho}$ is continuous, irreducible, and odd, then there exists a modular eigenform $f \in S_k(\Gamma_1(N))$, for some weight $k$ and level $N$, such that $\overline{\rho} \simeq \overline{\rho_{f, p}}$.
\end{conjecture}

This statement has a natural generalization for totally real fields:

\begin{conjecture}[Weak Serre conjecture] \label{weakserre}
If $F$ is a totally real field and $\overline{\rho}: G_F \to \GL_2(\Fpbar)$ is continuous, irreducible, and totally odd, then it is modular.
\end{conjecture}

While this conjecture is already very powerful, one generally wants to know in what weights and levels to look for a modular form giving rise to $\overline{\rho}$.  In fact, Serre gave such a strengthened version of his own conjecture, in which he specified a minimal weight $k(\overline{\rho})$ and level $N(\overline{\rho})$ such that there should exist a modular eigenform $f \in S_{k(\overline{\rho})}(\Gamma_1(N(\overline{\rho})))$ with $\overline{\rho} \simeq \overline{\rho_{f,p}}$.  We shall not give the explicit formulae for $k(\overline{\rho})$ and $N(\overline{\rho})$ here, but the reader will be able to extract them from our statement of a generalized conjecture later on.

\subsection{Serre weights} \label{serrewtsection}
One of the most basic and useful facts about mod $p$ representation theory, and one which is responsible for much of the difference in flavor between it and representation theory in characteristic zero, is the following.  A proof may be found in \cite{EdixhovenFLT}.

\begin{proposition} \label{basicfact}
Let $G$ be a profinite group, let $P \subset G$ be a normal pro-$p$-group, and let $\tau: P \to \GL(V)$ be a continuous finite-dimensional representation of $P$ on an $\Fpbar$-vector space $V$.  Let $V^P = \{ v \in V: \forall a \in P, \tau(a)v = v \}$.  Then $V^P \neq \{ 0 \}$.  Moreover, if $\tau$ is irreducible, then $V^P = V$.
\end{proposition}
\begin{proof}
Clearly we may restrict to the case where $V$ is irreducible as a $G$-module.  Since $\tau$ is continuous and hence has finite image, we in fact have $\tau: G \to \GL(W)$, where $W$ is a finite-dimensional vector space over a finite extension of $\F_p$.  Since $P$ is normal in $G$, clearly $W^P \subset W$ is a sub-$G$-module.  Since $P$ is a pro-$p$-group, every non-trivial orbit of the $P$-action on $W$ must have cardinality divisible by $p$.  However, the cardinality of $W$ is itself divisible by $p$.  Hence, $| W^P | \neq 1$, and thus $W^P = W$.
\end{proof}

\begin{definition} \label{srwtdef}
Let $F$ be a number field, let $v$ be a place of $F$, and let $k_v = \mathcal{O}_F / v$ be the residue field at $v$.
\begin{enumerate}
\item A {\emph{Serre weight}} is an irreducible $\Fpbar$-representation of $\GL_2(\mathcal{O}_F / p)$.
\item A {\emph{local Serre weight at $v$}} is an irreducible $\Fpbar$-representation of $\GL_2(k_v)$.
\end{enumerate}
\end{definition}

Since $\GL_2(\mathcal{O}_F / p)$ is a finite group, there are only finitely many Serre weights for any number field $F$.  Moreover, suppose that the ideal $p \mathcal{O}_F$ decomposes into prime factors as $p \mathcal{O}_F = \p_1^{e_1} \p_2^{e_2} \cdots \p_r^{e_r}$.  Then by the Chinese remainder theorem, $\GL_2(\mathcal{O}_F / p) = \GL_2(\mathcal{O}_F / \p_1^{e_1}) \times \cdots \times \GL_2(\mathcal{O}_F / \p_r^{e_r})$.  For each $1 \leq i \leq r$, let $k_i$ denote the residue field $\mathcal{O}_F / \p_i$.  Since the kernel of the natural projection
$$ \GL_2(\mathcal{O}_F / p) \to \GL_2(k_1) \times \cdots \times \GL_2(k_r)$$
is a $p$-group, all Serre weights factor through this projection by Proposition \ref{basicfact}.  It follows that all Serre weights have the form
$$ \sigma = \bigotimes_{v | p} \sigma_v,$$
where $\sigma_v$ is a local Serre weight at $v$.  The representation theory of $\GL_2$ of a finite field is well known \cite{Green}, and the distinct local Serre weights at $v$ are precisely the following:
$$ \sigma_v = \bigotimes_{\tau: k_v \hookrightarrow \Fpbar} \det\nolimits^{w_\tau} \tensor (\Sym^{r_\tau} k_v^2 \tensor_{k_v, \tau} \Fpbar).$$

Here, $k_v^2$ is the standard action of $\GL_2(k_v)$ on a two-dimensional vector space over $k_v$, whereas $ 0 \leq r_\tau \leq p - 1$ for all $\tau$ and $0 \leq w_\tau \leq p - 1$ for all $\tau$, with the stipulation that the $w_\tau$ are not all $p - 1$.  

What is the connection between these Serre weights and the modularity of Galois representations?  Observe, first of all, that if $F = \Q$, then the Serre weights are just $\det^w \tensor \Sym^r \Fpbar^2$, where $0 \leq r \leq p - 1$ and $0 \leq w \leq p - 2$.  Let $\overline{\rho}: G_{\Q} \to \GL_2(\Fpbar)$ be a mod $p$ Galois representation, and let $\T$ be the Hecke algebra generated over $\Z$ by the Hecke operators $T_l$ for $l \nmid p N(\overline{\rho})$.  Define the maximal ideal $\m_{\overline{\rho}} \subset \T$ to be the kernel of the map
\begin{eqnarray*}
\T & \to & \Fpbar \\
T_l & \mapsto & \tr \overline{\rho}(\Frob_l).
\end{eqnarray*}

The following was observed by Ash and Stevens \cite{ASCrelle}.  

\begin{proposition} \label{ashstevens}
Let $k \geq 2$.  A Galois representation $\overline{\rho}: G_{\Q} \to \GL_2(\Fpbar)$ is modular of level $N$ and weight $k$ if and only if $H^1(\Gamma_1(N), \Sym^{k-2} \Fpbar^2)_{\m_{\overline{\rho}}} \neq 0$.
\end{proposition}

We remark that by Theorem 3.4 of \cite{Edixhoven}, any collection of eigenvalues of $\T$ for which there exists an eigenform of some weight has an eigenform of weight at most $p + 1$.  Thus we do not lose generality by concentrating on forms associated to Serre weights.  The twist by a power of the determinant that can occur in the Serre weight corresponds to Nebentypus.

\subsection{Modularity} \label{modularitysection}
Inspired by the observations above, we will formulate a new definition of modularity.  We use quaternionic Shimura curves in place of modular curves.

Recall that $F$ is our fixed totally real field and that $d = [F: \Q]$.  Let $B / F$ be a quaternion algebra that splits at exactly one real place and at all places above $p$ (note that we make no conditions at places away from $p$, so there are no parity issues).  In other words, we are able to fix isomorphisms
\begin{eqnarray*}
B \tensor \R & \simeq & M_2(\R) \times \mathbb{H}^{d-1} \\
B \tensor \Q_p & \simeq & M_2(F \tensor \Q_p).
\end{eqnarray*}

Denote by $\A$ the adeles of $\Q$, and if $S$ is a finite set of places of $\Q$ then we set $\A^S$ to be the adeles away from $S$.
Define the group $G = \mathrm{Res}_{F / \Q} B^\ast$, and let $U \subset G(\A^\infty)$ be an open compact subgroup of the form $U = U_p \times U^p$, where $U^p \subset G(\A^{\infty, p})$ and
$$
U_p = \mathrm{ker} \left( \prod_{v | p} \GL_2(\mathcal{O}_{F_v}) \to \prod_{v | p} \GL_2(k_v) \right).$$

Then the Shimura curve
$$ M_U (\C) = G(\Q) \backslash (G(\A^\infty) \times (\C - \R)) / U$$
has a model over $F$.  Set $V \subset G(\A^\infty)$ to be the following open compact subgroup:
$$ V =  \left( \prod_{v | p} \GL_2(\mathcal{O}_{F_v}) \right) \times U^p.$$

It follows from a simple modification of an argument of Carayol \cite{Carayol} that if $U^p$ is sufficiently small (see \cite{thesispaper} for the precise definition of ``sufficiently small'') then the natural map $M_U \to M_V$ is a Galois cover of Shimura curves with Galois group $V / U \simeq \prod_{v|p} \GL_2(k_v)$.

\begin{definition}
Let $F$ be a totally real field, let $\sigma$ be a Serre weight, and let $\overline{\rho}: G_F \to \GL_2(\Fpbar)$ be a mod $p$ Galois representation.  We say that $\overline{\rho}$ is {\emph{modular of weight $\sigma$}} if there exist a quaternion algebra $B / F$ and an open compact subgroup $U \subset G(\A^\infty)$ as above such that
\begin{equation} \label{modularitydef}
 \overline{\rho} \subset (\mathrm{Pic}^0(M_U)[p] \tensor \sigma)^{\prod_{v|p} \GL_2(k_v)}.
 \end{equation}
\end{definition}

The reader is referred to the introduction of \cite{BDJ} and to Propositions 2.5 and 2.10 of the same paper for a discussion of why the naive notion of modularity (a Galois representation is modular if it arises from a Hilbert modular form of weight $(k_1, \dots, k_d)$) is insufficient and of the connection between this naive modularity and the notion of being modular of a certain Serre weight.

\begin{remark}
It is not hard to show, using the Eichler-Shimura relation, that the condition (\ref{modularitydef}) is equivalent to the following:
\begin{equation*}
\overline{\rho}^\vee = \mathrm{Hom}(\overline{\rho}, \Fpbar) \subset H^1_{\acute{e} t}(M_V \tensor \overline{\Q} , \mathcal{L}_{\sigma}),
\end{equation*}
where $\mathcal{L}_{\sigma}$ is the mod $p$ local system associated to $\sigma$.  While this fact is well-known, the author does not know of a clear proof in the literature.  See Proposition 2.6 of \cite{thesispaper}.
\end{remark}

\begin{remark}
The reader may wonder why we assume throughout that $F$ is a totally real field.  Why do we not work with arbitrary number fields?  This question is strengthened by the fact that the proofs of many of the results about Galois representations $\overline{\rho}: G_F \to \GL_2(\Fpbar)$ that are described below do not use global information about $\overline{\rho}$, but only local information about the restrictions of $\overline{\rho}$ to the decomposition subgroups at various places.  It is very difficult to study modular Galois representations over arbitrary number fields $F$ because of the lack of nice algebraic varieties over $F$ that store automorphic data and in whose cohomology we could look for our Galois representations.  The theory of Shimura varieties over totally real fields has no good analogue over general number fields.  A number of mathematicians have studied the modularity for quadratic imaginary fields: see, for example, \cite{TaylorImagI}, \cite{TaylorImagII}, and \cite{BergerHarcos}.  In this case one can translate the problem to Siegel modular surfaces.  See \cite{BergerKlosin} and the end of \cite{CalegariGeraghty} for modularity lifting theorems in the quadratic imaginary case.
\end{remark}

\subsection{Serre's weight conjecture} \label{statementofconj}
For every place $v$ of $F$ dividing $p$, denote by $k_v$ the corresponding residue field and let its cardinality be $q_v = p^{f_v}$.  Recall that $G_{F_v}$ can be embedded non-canonically into $G_F$ as follows.  For every finite extension $L/F$, we choose a place $v_L$ of $L$ such that $v_L | v$ and such that if $L \subset L^\prime$, then $v_{L^\prime} | v_L$.  Now let $G_v = \{ \alpha \in G_F : \forall L/F, \alpha(v_L) = v_L \}.$  This $G_v$ is called a decomposition subgroup at $v$, and it is easy to show that $G_v \simeq G_{F_v}$.  If we replace the system $\{v_L\}_L$ by a different compatible system of places, we will get a subgroup conjugate to $G_v$.  Inside $G_v$, we have the inertia subgroup $I_v$ consisting of all $\alpha \in G_v$ such that for each $L / F$, the induced automorphism of $\mathcal{O}_L / v_L$ is trivial.  The wild inertia $P_v$ is the pro-$p$-Sylow subgroup of $I_v$.  Note that the isomorphism $G_v \simeq G_{F_v}$ induces isomorphisms $I_v \simeq \Gal(\overline{\Q}_p / F_v^{nr})$ and $P_v \simeq \Gal(\overline{\Q}_p / F_v^{tr})$, where $F_v^{nr}$ and $F_v^{tr}$ are the maximal unramified and maximal tamely ramified extensions of $F_v$, respectively.  In particular, since $F_v^{tr}$ is a Galois extension of $F_v$, we see that $P_v$ is a normal subgroup of $G_v$.  We have an exact sequence
\begin{equation} \label{ivgv}
 1 \to I_v \to G_v \to \Gal(\Fpbar / k_v) \to 1.
 \end{equation}

Now, let $r \geq 1$ and consider an embedding of fields $\tau: \F_{p^r} \hookrightarrow \Fpbar$.  Fix a uniformizer $\pi_v \in \mathcal{O}_{F_v}$ and define $\psi_\tau: I_v \to \Fpbar^\ast$ to be the following composition of homomorphisms:
$$ I_v \simeq \Gal(\overline{\Q}_p / F_v^{nr}) \to \Gal(F_v^{nr}(\sqrt[p^r - 1]{\pi_v}) / F_v^{nr} ) \simeq \F_{p^r}^{\ast} \stackrel{\tau}{\to} \Fpbar^\ast.$$

Let $\overline{\rho}: G_F \to \GL_2(\Fpbar)$ be a mod $p$ Galois representation.  Then $P_v$ acts trivially on $(\overline{\rho}_{| G_v})^{ss}$ by Proposition \ref{basicfact}.  It follows that 
$((\overline{\rho}_{| G_v})^{ss})_{| I_v}$ factors through a representation of the abelian group $I_v / P_v$ and thus is a sum of characters $\vphi \oplus \vphi^\prime$.  If $\Frob_v$ is a lift to $G_v$ via (\ref{ivgv}) of the map $x \mapsto x^{q_v}$, which is a topological generator of $\Gal(\Fpbar / k_v)$, then $\Frob_v$ acts on $((\overline{\rho}_{| G_v})^{ss})_{| I_v}$ by conjugation.  Therefore, $\{ \vphi, \vphi^\prime \} = \{ (\vphi)^{q_v}, (\vphi^\prime)^{q_v} \}$.  Let $[k_v^\prime : k_v] = 2$.  We have two possibilities:

\begin{enumerate}
\item If $\overline{\rho}_{| G_v}$ is reducible, then $\overline{\rho}_{| I_v} \sim \left( \begin{array}{cc} \vphi & \ast \\ 0 & \vphi^\prime \end{array} \right)$, where $\vphi$ and $\vphi^\prime$ each factor through $k_v$.
\item If $\overline{\rho}_{| G_v}$ is irreducible, then $\overline{\rho}_{| I_v} \sim \left( \begin{array}{cc} \vphi & 0 \\ 0 & \vphi^\prime \end{array} \right)$, where $\vphi$ and $\vphi^\prime$ factor through $(k_v^\prime)^\ast$ and we have $\vphi^\prime = \vphi^{q_v}$ and $\vphi = (\vphi^\prime)^{q_v}$.
\end{enumerate}

We are finally ready to state Serre's conjecture.  Given a mod $p$ Galois representation $\overline{\rho}: G_F \to \GL_2(\Fpbar)$, our aim is to define a set $W(\overline{\rho})$ of Serre weights and then conjecture:

\begin{conjecture}[Strong Serre conjecture] \label{strongserre}
If $\overline{\rho}$ is continuous, irreducible, and totally odd, then it is modular.  Moreover, if $\sigma$ is a Serre weight, then $\overline{\rho}$ is modular of weight $\sigma$ if and only if $\sigma \in W(\overline{\rho})$.
\end{conjecture}

A crucial property of the set $W(\overline{\rho})$ is that it is defined locally.  This means that for each place $v | p$ we will specify a set $W_v(\overline{\rho})$ of local Serre weights at $v$ (recall Definition \ref{srwtdef}) and then define
\begin{equation}
W(\overline{\rho}) = \left\{ \sigma = \bigotimes_{v | p} \sigma_v  :  \forall v | p, \sigma_v \in W_v(\overline{\rho}) \right\}.
\end{equation}

Given a place $v | p$, let $e_v$ be the ramification index of $v | p$, so that $[F_v : \Q_p] = e_v f_v$.  Let $S_v$ be the collection of all field embeddings $k_v \hookrightarrow \Fpbar$.  The definition of $W_v(\overline{\rho})$ involves several cases, corresponding to the cases above.

\begin{enumerate}
\item If $\overline{\rho}_{| G_v}$ is reducible and semisimple (i.e. the direct sum of two one-dimensional representations) then a local Serre weight at $v$
\begin{equation} \label{lclsrwt}
\sigma_v = \bigotimes_{\tau \in S_v} \left( \det\nolimits^{w_\tau} \tensor \Sym^{r_\tau} k_v^2 \tensor_{k_v, \tau} \Fpbar \right)
\end{equation}
is contained in $W_v(\overline{\rho})$ if and only if there exists a subset $A \subset S_v$ and an integer $0 \leq \delta_\tau \leq e_v - 1$ for each $\tau \in S_v$, such that
\begin{equation} \label{smcredcase}
\overline{\rho}_{| I_v} \sim \left( \begin{array}{cc} \prod_{\tau \in A} \psi_{\tau}^{w_\tau + r_\tau + 1 + \delta_\tau} \prod_{\tau \not\in A} \psi_{\tau}^{w_\tau + e_v - 1 - \delta_\tau} & 0 \\
0 & \prod_{\tau \in A} \psi_{\tau}^{w_\tau + e_v - 1 - \delta_\tau} \prod_{\tau \not\in A} \psi_{\tau}^{w_\tau + r_\tau + 1 + \delta_\tau} \end{array} \right).
\end{equation}

\item If $\overline{\rho}_{| G_v}$ is irreducible, then $\sigma_v$ as in (\ref{lclsrwt}) is contained in $W_v(\overline{\rho})$ if and only if for each $\tau \in S_v$ there exists an integer $0 \leq \delta_\tau \leq e_v - 1$ and a lift $\tilde{\tau} : k_v^\prime \hookrightarrow \Fpbar$ of $\tau$ (recall that $k_v^\prime$ is a quadratic extension of $k_v$) such that
$$ \overline{\rho}_{| I_v} \sim \left( \begin{array}{cc} \vphi & 0 \\ 0 & \vphi^{q_v} \end{array} \right),$$
where
\begin{equation} \label{smcirredcase}
\vphi = \prod_{\tau \in S_v} \psi_{\tilde{\tau}}^{(q_v + 1) w_\tau + r_\tau + 1 + \delta_\tau + q_v (e_v - 1 - \delta_\tau)}.
\end{equation}
\item If $\overline{\rho}_{| G_v}$ is indecomposable, i.e. reducible but not semisimple, then 
$$ \overline{\rho}_{| G_v} \sim \left( \begin{array}{cc} \vphi & \ast \\ 0 & \vphi^\prime \end{array} \right),$$
where $\ast$ corresponds to an element $c_{\overline{\rho}} \in \mathrm{Ext}^1(\vphi^\prime, \vphi) \simeq H^1(G_{v}, \vphi (\vphi^\prime)^{-1})$.  For each local Serre weight $\sigma_v$ and each subset $A \subset S_v$ one defines a subspace $L_{A, \sigma_v} \subset H^1(G_{v}, \vphi (\vphi^\prime)^{-1})$.  The definition of $L_{A, \sigma_v}$ is intricate, and we omit it here.  Then a local Serre weight as in (\ref{lclsrwt}) is contained in $W_v(\overline{\rho})$ if and only if there exist a subset $A \subset S_v$ and an integer $0 \leq \delta_\tau \leq e_v - 1$ for each $\tau \in S_v$ such that $((\overline{\rho}_{| G_v})^{ss})_{| I_v}$ has the form of (\ref{smcredcase}) and in addition $c_{\overline{\rho}} \in L_{A, \sigma_v}$.
\end{enumerate} 

As we mentioned above, this conjecture was first stated in the case of $F = \Q$ by Serre several decades ago, in a slightly different language.  We will now indicate how to extract the minimal weight $k(\overline{\rho})$ from the conjecture formulated above.  It follows from Proposition \ref{ashstevens} and the definition of modularity that $\overline{\rho}$ arises from a modular form of weight $2 \leq k \leq p + 1$ (and trivial Nebentypus) precisely when it is modular of the Serre weight $\Sym^{k-2} \Fpbar^2$.  Furthermore, it is easily shown by an argument with Eisenstein ideals that, for arbitrary $k \geq 2$, we have $H^1(\Gamma_1(N), \Sym^{k-2} \Fpbar^2)_{\mathfrak{m}_{\overline{\rho}}} \neq 0$ if and only if $H^1(\Gamma_1(N), W)_{\mathfrak{m}_{\overline{\rho}}} \neq 0$ for some Jordan-H\"{o}lder constituent $W$ of $\Sym^{k-2} \Fpbar$.  Thus, the minimal weight $k(\overline{\rho})$ conjectured by Serre is just the minimal $k \geq 2$ such that $\overline{\rho}$ is modular (in the sense of Conjecture \ref{strongserre}) of some Jordan-H\"{o}lder constituent of $\Sym^{k-2} \Fpbar^2$.  This $k(\overline{\rho})$ is easily computed using the following decomposition.

\begin{lemma}
Let $m > p - 1$.  Then the following holds, where ``ss'' denotes semisimplification:
$$ (\Sym^m \Fpbar^2)^{ss} = (\det \tensor \Sym^{m - p - 1} \Fpbar^2)^{ss} \oplus \Sym^r \Fpbar^2 \oplus \det\nolimits^r \tensor \Sym^{p-1-r} \Fpbar^2,$$
where $r$ is the unique integer in the range $0 \leq r \leq p - 2$ such that $r \equiv m$ modulo $p - 1$.
\end{lemma}

The first generalization of Serre's conjecture beyond the original case of $F = \Q$, to the case of totally real fields $F$ in which $p$ is unramified, was by Buzzard, Diamond, and Jarvis \cite{BDJ} and circulated for nearly a decade before their paper appeared in print.  A conjecture for arbitrary totally real fields $F$ but semisimple $\overline{\rho}_{| G_v}$ (i.e. the first two of the three cases above) was made by the author; see \cite{ramifiedpaper}, Theorems 2.4 and 2.5.  At around the same time, Herzig \cite{Florianthesis} made a conjecture for $n$-dimensional representations $\overline{\rho}: G_\Q \to \GL_n(\Fpbar)$ for arbitrary $n$, under the assumption that $\overline{\rho}_{| G_p}$ is semisimple.  These conjectures were later restated by other authors \cite{BLGG} \cite{EGH} to cover arbitrary $\overline{\rho}$ at the cost of becoming less explicit: if $\sigma$ is a local Serre weight at $v | p$, then $\sigma \in W_v(\overline{\rho})$ is conjectured to be equivalent to the existence of a $p$-adic lift of $\overline{\rho}$ with some specified local properties.  Such opacity appears already in the indecomposable case of the conjecture of \cite{BDJ}, in the form of the spaces $L_{A, \sigma_v}$.

\subsection{The level in Serre's conjecture} \label{levelsection}
Before proceeding, we will say a few words about the level in Serre's conjecture.  For this, the notion of the Artin conductor of a Galois representation is crucial.  Roughly speaking, the Artin conductor $\mathfrak{n}(\overline{\rho})$ of $\overline{\rho} : G_F \to \GL_2(\Fpbar)$ is an ideal $\mathfrak{n}({\overline{\rho}}) = \prod_v v^{a_{v}}$ of $\mathcal{O}_F$, where $v$ runs over the finite places of $F$ and $a_v$ measures the ramification of $\overline{\rho}$ at $v$.  If $\overline{\rho}$ is unramified at $v$, i.e. if $I_v \subset \ker \overline{\rho}$, then $a_v = 0$.  Otherwise, the exponent $a_v$ reflects how far down the upper ramification filtration of $I_v$ one has to go to find a subgroup contained in $\ker \overline{\rho}$.  More precisely, let $L / F$ be a finite Galois extension such that $\overline{\rho}$ factors through $\Gal(L/F)$; this exists because $\overline{\rho}$ is continuous and $G_F$ is compact, hence $\overline{\rho}$ has finite image.  if $v \nmid p$, then let $v_L$ be a place of $L$ lying above $v$, and for each $i \geq 0$ let $G_i \subset G_{v_L} \simeq \Gal(\overline{L_{v_L}} / L_{v_L})$ be the $i$-th ramification subgroup; see, for instance, Chapter IV of \cite{SerreCL} for the definition.  If $V_{\overline{\rho}}$ is a two-dimensional $\Fpbar$-vector space on which $\overline{\rho}$ acts, then we define

\begin{equation*}
a_v = \sum_{i = 0}^{\infty} \frac{1}{[I_{v_L} : G_i]} \dim_{\Fpbar} (V_{\overline{\rho}} / V_{\overline{\rho}}^{G_i}),
\end{equation*}
where $V_{\overline{\rho}}^{G_i}$ is the subspace of $G_i$-invariants.  More details about the Artin conductor may be found in Chapter VI of \cite{SerreCL}.

When $F = \Q$, any modular Galois representation $\overline{\rho}: G_\Q \to \GL_2(\Fpbar)$ arises from a modular form of weight $N$ that is prime to $p$, and Serre specified a minimal such weight $N(\overline{\rho})$: this is just the prime-to-$p$ part of the Artin conductor of $\overline{\rho}$.  If $F$ is an arbitrary totally real field, we are no longer so lucky.  We can no longer expect $\overline{\rho}$ to always arise in level prime to $p$ (this would correspond to $U_p = \prod_{v | p} \GL_2(\mathcal{O}_{F_v})$ in Section \ref{modularitysection}); see the introduction of \cite{BDJ} for a discussion of why not.  However, the prime-to-$p$ part of the level may conjecturally always be taken to be the prime-to-$p$ part of $\mathfrak{n}(\overline{\rho})$, which means that we may take $U^p = \prod_{v \nmid p} U_1(v^{a_v})$, where the group $U_1(v^{a_v}) \subset \GL_2(\mathcal{O}_{F_v})$ consists of all matrices $\left( \begin{array}{cc} a & b \\ c & d \end{array} \right) \in \GL_2(\mathcal{O}_{F_v})$ such that $a - 1 \in v^{a_v}$ and $c \in v^{a_v}$.

\section{Weak Serre implies strong Serre}
The ``weak'' version of Serre's modularity conjecture (Conjecture \ref{weakserre}) is actually a very strong statement.  It has been proved only in the case $F = \Q$.  This was achieved by Khare, Wintenberger, and Kisin \cite{KWfirst}, \cite{KWsecond}, \cite{Kisin}, and we will sketch some of their methods below.  As we will see at the end of these notes, some of these methods fail crucially whenever $F$ is a totally real field with $[F : \Q] > 1$, so that a major new idea is needed for any substantial further progress on the conjecture.  

In the meantime, much research in the area has focused on proving, in various settings, that the weak version of Serre's conjecture (Conjecture \ref{weakserre}) implies the strong version (Conjecture \ref{strongserre}), in other words that if $\overline{\rho}$ is modular, then it is modular of precisely the predicted Serre weights.  In the case $F = \Q$, this fact has essentially been known since the late 1970's except for a few cases where $p = 2$ and was an important input in the proof of Serre's conjecture (the stubborn cases with $p = 2$ were also settled by Khare, Wintenberger, and Kisin's work).  It was proved by Deligne for $\overline{\rho}_{| G_p}$ reducible and by Fontaine for $\overline{\rho}_{| G_p}$ irreducible.  Fontaine's work was never published, and a (somewhat different) proof of the theorem first appeared in print in \cite{Edixhoven}.

We will now mention some of the ``weak Serre implies strong Serre'' theorems that have been proved in recent years.
The conjecture of \cite{BDJ}, for $F$ in which $p$ is unramified, was proved by Gee \cite{GeeBDJ} for most Serre weights by deformation-theoretic methods; see also \cite{thesispaper} for a more geometric proof of most cases of one direction of this conjecture for locally irreducible $\overline{\rho}$.  We will say more about the methods of these papers in the remainder of this section.  The remaining cases of the Buzzard-Diamond-Jarvis conjecture were attacked in a series of papers by Gee and coauthors, until it was finally proved completely in \cite{GeeKisin}.  The results of \cite{thesispaper} were extended in \cite{ramifiedpaper} to cases where $p$ ramifies in $F$.  Moreover, the conjecture of \cite{ramifiedpaper} was proved, for most cases where $p$ is totally ramified in $F$, by Gee and Savitt \cite{GeeSavitt}.  More cases in the related, but not equivalent, unitary setting were resolved in \cite{GSL}.  Some non-totally ramified cases with $e = f = 2$ were addressed by R. Smith in his Ph.D. thesis at the University of Arizona, but it seems that new ideas are needed to make substantial further progress.

\subsection{A sketch of Gee's argument}
The claim that weak Serre implies strong Serre consists, of course, of two claims in opposite directions:
\begin{enumerate}
\item If $\overline{\rho}$ is modular of weight $\sigma$, then $\sigma \in W(\overline{\rho})$.
\item If $\overline{\rho}$ is modular of some weight and $\sigma \in W(\overline{\rho})$, then $\overline{\rho}$ is modular of weight $\sigma$.
\end{enumerate}

The most successful method for proving ``weak Serre implies strong Serre'' has been that of relating the modularity of $\overline{\rho}$ to the existence of lifts of $\overline{\rho}$ with some specific local properties and then using $p$-adic Hodge theory to investigate the existence of such lifts.  

An important breakthrough was Gee's paper \cite{GeeBDJ}, which proved the following result.  Its statement involves the following definition: a local Serre weight at $v$ is said to be regular if it is of the form (\ref{lclsrwt}) with $1 \leq r_\tau \leq p - 3$ for all $\tau \in S_v$.  A Serre weight $\sigma = \bigotimes_{v | p} \sigma_v$ is called regular if all the $\sigma_v$ are regular.

\begin{theorem} \label{geebdjthm}
Suppose that $p \geq 5$ is unramified in the totally real field $F$, that $\overline{\rho}: G_F \to \GL_2(\Fpbar)$ is modular of some Serre weight, and that $\overline{\rho}_{| G_{F(\zeta_p)}}$ is irreducible.  Suppose that $\sigma$ is a regular Serre weight.  If $\overline{\rho}$ is modular of weight $\sigma$, then $\sigma \in W(\overline{\rho})$.  Conversely, if $\sigma \in W(\overline{\rho})$ and some further technical conditions are satisfied at places $v | p$ where $\overline{\rho}_{| G_v}$ is reducible and $\sigma_v$ arises from $A = S_v$ or $A = \varnothing$ in the recipe of Section \ref{statementofconj}, then $\overline{\rho}$ is modular of weight $\sigma$.
\end{theorem}

We will give a very brief sketch of part of the argument of \cite{GeeBDJ} to illustrate the method; the reader is referred to that paper (and the papers cited in it!) for further details.  It should be noted that Gee works with a different notion of modularity than the one given above; he uses definite quaternion algebras, rather than indefinite ones.  It is not possible to translate theorems directly from one setting to the other, but his local arguments can be translated.  In this section, we will assume that $p$ is unramified in the totally real field $F$.  Let $\overline{\rho}: G_F \to \GL_2(\Fpbar)$ be a continuous, irreducible, and totally odd mod $p$ Galois representation.  Recall that in Section \ref{statementofconj} we defined a set $W_v(\overline{\rho})$ of local Serre weights at $v$ for each $v | p$.

Let $v | p$; for the purposes of this section, we will say that a $p$-adic representation $\eta_v$ of $\GL_2(k_v)$ is {\emph{good}} if it is either an irreducible principal series or supercuspidal; in other words, $\eta_v$ is any irreducible $p$-adic representation of $\GL_2(k_v)$ that is not one-dimensional or special.  We regard $\eta_v$ as a representation of the group $\GL_2(\mathcal{O}_{F_v})$ via the obvious inflation.  Then $\eta_v$ has an associated inertial type $\tau_v$, namely a $p$-adic representation of $I_v$ with the property that for any irreducible $p$-adic representation $\pi$ of $\GL_2(F_v)$, we have $\eta_v \subset \pi_{|GL_2(\mathcal{O}_{F_v})}$ if and only if $\mathrm{LLC}(\pi)_{| I_v} \simeq \tau_v$.  Here $\mathrm{LLC}(\pi)$ is the Weil-Deligne representation associated to $\pi$ by the local Langlands correspondence.  See Henniart's appendix to \cite{BMH} for an exposition of the theory of types for $\GL_2$.

\begin{proposition} [\cite{GeeBDJ}, Lemma 2.1.4] \label{geelemmaaa}
Let $\overline{\rho}$ be as above, and for each $v | p$ let $\eta_v$ be a good representation as above.  Then $\overline{\rho}$ is modular of some Serre weight $\sigma \in \mathrm{JH}(\bigotimes_{v | p} (\eta_v \tensor \Fpbar))$ if and only if $\overline{\rho}$ has a modular $p$-adic lift $\rho: G_F \to \GL_2(\overline{\Q}_p)$ such that for each $v | p$, the restriction $\rho_{| G_v}$ is potentially Barsotti-Tate (i.e. potentially crystalline with Hodge-Tate weights $(0,1)$) and $\mathrm{WD}(\rho_{| G_v})_{| I_v} = \tau_v$.
\end{proposition}

Here, and subsequently, we write $JH(V)$ for the set of Jordan-H\"{o}lder constituents of a representation $V$, whereas $\mathrm{WD}(\rho_{| G_v})$ denotes the Weil-Deligne representation corresponding to the local Galois representation $\rho_{| G_v}$.  The reader is referred to the classic article \cite{Tate} for the correspondence between Galois and Weil-Deligne representations.

Suppose that we know how to prove the first of the two claims at the beginning of this section, namely that if $\overline{\rho}$ is modular of weight $\sigma$, then $\sigma \in W(\overline{\rho})$.  Assuming that, here is a strategy for proving the second claim.  Let $\sigma \in W(\overline{\rho})$ be a Serre weight.  If it is regular, then for each $v | p$ there exists a good $\eta_v$ as above such that $\mathrm{JH}(\tensor_{v|p} (\eta_v \tensor \Fpbar)) \cap W(\overline{\rho}) = \{ \sigma \}$.  Then by Proposition \ref{geelemmaaa} it suffices to find a modular lift $\rho$ with the properties specified in the statement of that proposition.

The most daunting aspect of coming up with a lift $\rho$ of $\overline{\rho}$ that satisfies the conditions of Proposition \ref{geelemmaaa} is clearly that of showing that the $\rho$ we have constructed is modular.  Fortunately, Gee's adaption of a modularity lifting theorem of Kisin comes to the rescue.  This is the first of many close connections that we will see in these lectures between Serre's modularity conjecture and modularity lifting theorems.

\begin{proposition} \label{geekisinprop}
Suppose that the hypotheses of Theorem \ref{geebdjthm} hold and that $\rho: G_F \to \GL_2(\overline{\Q}_p)$ is a lift of $\overline{\rho}$ such that $\rho_{| G_v}$ is potentially Barsotti-Tate and $\mathrm{WD}(\rho_{| G_v})_{| I_v} = \tau_v$ for each $v | p$.  Suppose that there exists a cuspidal automorphic representation $\pi$ of $\GL_2(\A_F)$ such that for every $v | p$, the local Galois representation $\rho_{\pi, v}$ is potentially ordinary if and only if $\rho_{| G_v}$ is potentially ordinary.  Then $\rho$ is modular.
\end{proposition}

Note that the hypothesis on $\rho_{| G_{F(\zeta_p)}}$ in the statement of Theorem \ref{geebdjthm} is common in modularity lifting theorems \`{a} la Kisin, and this is the point in the proof where it is necessary.

Now we need to construct a lift $\rho$ satisfying the conditions of Proposition \ref{geekisinprop}.  The theory of Breuil modules allows us to translate local conditions on Galois representations into linear-algebraic data.

Let $k$ be a finite field of characteristic $p > 2$, let $W(k)$ be the associated ring of Witt vectors, and let $K_0 = W(k)[1/p]$ be its fraction field.  Let $K / K_0$ be a totally tamely ramified Galois extension of degree $e$.  Let $B \subset K_0$ be a subfield such that there exists a uniformizer $\pi \in \mathcal{O}_K$ satisfying $\pi^e \in B$.  Choose such a $\pi$.  Let $2 \leq k \leq p - 1$ be an integer; this conflict of notation is standard and will produce no confusion.  Let $E / \F_p$ be a finite extension.  The category $\mathrm{BrMod}_{dd, B}^{k-1}$ of Breuil modules with descent data has as objects quintuples $(M, M_{k-1}, \vphi_{k-1}, N, \hat{g})$ such that:
\begin{enumerate}
\item $M$ is a finitely generated $(k \tensor_{\F_p} E)[u]/u^{ep}$-module that is free over $k[u]/u^{ep}$.
\item $M_{k-1}$ is a submodule such that $u^{e(k-1)} M \subset M_{k-1}$.
\item $\vphi_{k-1}: M_{k-1} \to M$ is an $E$-linear and Frobenius-semilinear homomorphism whose image generates $M$ as a $(k \tensor_{\F_p} E)[u]/u^{ep}$-module.  Frobenius-semilinear in this case means that if $a \in k[u]/ u^{ep}$ and $m \in M_{k-1}$, then $\vphi_{k-1}(am) = a^p \vphi_{k-1}(m)$.
\item $N: M \to uM$ is a $(k \tensor_{\F_p} E)$-linear map satisfying
\begin{enumerate}
\item $N(um) = uN(m) - um$ for all $m \in M$.
\item $u^e N(M_{k-1}) \subset M_{k-1}$.
\item $\vphi_{k-1}(u^e N(m)) = - \frac{\pi^e}{p} N(\vphi_{k-1} (m))$ for all $m \in M_{k-1}$.
\end{enumerate}
\item For each $g \in \Gal(K/B)$, there is an additive bijection $\hat{g}: M \to M$ such that
\begin{enumerate}
\item Each $\hat{g}$ commutes with, $\vphi_{k-1}$, $M$, and the $E$-action.
\item $\hat{1}$ is the identity map, where $1 \in \Gal(K/B)$ is the identity automorphism.
\item $\hat{g} \circ \hat{h} = \widehat{g \circ h}$ for all $g, h \in \Gal(K/B)$.
\item $\hat{g}(a u^i m) = g(a)((g(\pi)/\pi)^i \tensor 1) u^i \hat{g}(m)$ for all $a \in k \tensor_{\F_p} E$, $m \in M$, and $i \geq 0$.  To make sense of $g(a)$, note that $k$ is the residue field of $K_0$, hence of $K$, and so is acted on by $\Gal(K/B)$.  We let $\Gal(K/B)$ act trivially on the second component of $k \tensor_{\F_p} E$.
\end{enumerate}
\end{enumerate}

The connection between Breuil modules and potentially Barsotti-Tate Galois representations is evidenced, for instance, by the fact that the category $\mathrm{BrMod}_{dd, B}^1$ is equivalent to the category of finite flat group schemes over $\mathcal{O}_K$ with an action of $E$ and descent data to $B$.  Gee proves that the existence of a lift $\rho$ which is potentially Barsotti-Tate at $v$ of inertial type $\tau_v$ is equivalent to the existence of a Breuil module satisfying certain conditions.  As we see from the definition above, Breuil modules with descent data are complicated objects but are very explicit, and one constructs the needed Breuil module by hand.

The proof for arbitrary (i.e. not necessarily regular) Serre weights follows the same lines, but the theory of Breuil modules, which itself is an extension of Fontaine-Laffaille theory, is not powerful enough.  Here one uses Liu's theory of Kisin modules.

\subsection{Modular weights are predicted ones: some algebraic geometry}  
In this section we will sketch how to prove that if $\overline{\rho}$ is modular of a Serre weight $\sigma$, then $\sigma \in W(\overline{\rho})$.  In order to illustrate the variety of methods applicable to this problem, we will give an algebraic-geometry argument following \cite{thesispaper} and \cite{ramifiedpaper}.  This method was used to obtain the earliest results in this direction, but it has turned out to be less effective than the deformation-theoretic and $p$-adic Hodge-theoretic techniques of which a flavor was given in the previous section.  The reader may, of course, find further details in \cite{thesispaper}.

In this section we will not impose such severe limitations on the ramification of $p$ in the totally real field $F$, but we will suppose that $\overline{\rho}_{| G_v}$ is irreducible for all $v | p$.  For each $v | p$, let $e_v$ be the ramification index of $F_v / \Q_p$.  Let $\sigma = \bigotimes_{v | p} \sigma_v$ be a Serre weight such that for each $v$ and each $\tau \in S_v$ we have $0 \leq r_\tau \leq p - e_v - 1$ (so in particular we are assuming here that $e_v \leq p - 1$).  

Now we will recall some notions from Sections \ref{serrewtsection} and \ref{modularitysection}.  Assume that $\overline{\rho}: G_F \to \GL_2(\Fpbar)$ is modular of weight $\sigma$.  By definition, this implies the existence of a quaternion algebra $B / F$, giving rise to an algebraic group $G$, and an open compact subgroup $V = \left( \prod_{v|p} \GL_2(\mathcal{O}_{F_v}) \right) \times U^p \subset G(\A^\infty)$ such that $H^1_{\acute{e}t}(M_V \tensor \overline{\Q}, \mathcal{L}_\sigma)_{\m_{\overline{\rho}}} \neq 0$.  For each $v | p$, let $U_1^{bal}(v) \subset \GL_2(\mathcal{O}_{F_v})$ be the subgroup of matrices whose reductions modulo $v$ are unipotent upper triangular, i.e. of matrices that are congruent to $\left( \begin{array}{cc} 1 & \ast \\ 0 & 1 \end{array} \right)$ modulo $v$.  Consider the open compact subgroup
$$ U_1^{bal}(p) = \left( \prod_{v | p} U_1^{bal}(v) \right) \times U^p \subset G(\A^\infty).$$

Fix a place $v | p$.  Let $D = W(k_v)$ be a ring of Witt vectors, let $K = F_v^{nr}$ be the fraction field of $D$, let $K^\prime = K(\sqrt[q_v - 1]{\pi_v})$ be a totally tamely ramified extension with $\Gal(K^\prime / K) \simeq k_v^\ast$, and let $D^\prime = \mathcal{O}_{K^\prime}$.  Then $M_{U_1^{bal}(p)}$ has an integral model over $D$, which we shall denote $\mathbf{M}_{U_1^{bal}(p)}$.  Moreover, $\mathbf{M}_{U_1^{bal}(p)} \times_{D} D^\prime$ has a well-behaved special fiber consisting of two smooth curves intersecting transversally at finitely many points.

Let $j: \Gal(K^\prime / K) \to \mathcal{O}_{F_v}^\ast / (1 + v)$ be the isomorphism induced by the Artin reciprocity map of local class field theory (normalized so as to send arithmetic Frobenius to uniformizers).  We have natural actions of $\GL_2(\mathcal{O}_{F_v})$ (coming from the $p$-component of $G(\A^\infty)$) and of $\Gal(K^\prime / K)$ on the special fiber of $\mathbf{M}_{U_1^{bal}(p)} \times_{D} D^\prime$, and Carayol (\cite{Carayol}, 10.3) shows that the action of $\gamma \in \Gal(K^\prime / K)$ is equal to that of $\left( \begin{array}{cc} j(\gamma)^{-1} & 0 \\ 0 & 1 \end{array} \right)$ and $\left( \begin{array}{cc} 1 & 0 \\ 0 & j(\gamma)^{-1} \end{array} \right)$, respectively, on the two components of the special fiber.

Recall that we are assuming that $\overline{\rho}$ is modular of a given Serre weight $\sigma$.  Let $B(k_v) \subset \GL_2(k_v)$ be the Borel subgroup of upper triangular matrices, and let $\theta: B(k_v) \to \Fpbar^\ast$ be a character such that $\sigma_v \in \mathrm{JH}(\mathrm{Ind}_{B(k_v)}^{\GL_2(k_v)} \theta)$.  Let $C$ be the N\'{e}ron model over $D^\prime$ of the curve $\mathrm{Pic}^0(\mathbf{M}_{U_1^{bal}(p)}) \times K^\prime$.  Then $C[p^\infty]$ is a $p$-divisible group, and the reduction $C[p^\infty] \tensor \T / \m_{\overline{\rho}}$ contains a finite piece $G_\theta$ on which the diagonal matrices in $\GL_2(\mathcal{O}_{F_v})$ act via the character $\theta$.  By the main result of \cite{BLR}, $G_\theta[\m_{\overline{\rho}}]_K$ is a direct sum of a finite number of copies of $\overline{\rho}$.  As in Section \ref{statementofconj} above, $\overline{\rho}_{| I_v}$ is a direct sum of two characters, $\vphi$ and $\vphi^\prime$, that satisfy $\vphi^{q_v} = \vphi^\prime$ and $(\vphi^\prime)^{q_v} = \vphi$.  We can pick out a subspace $H \subset G_\theta[\m_{\overline{\rho}}]_K$ of rank $q_v^2$ on which $\Gal(\overline{K}/K) \simeq I_v$ acts by the character $\vphi$.

Let $\F$ be a finite field, sufficiently large so that $\mathrm{im}(\overline{\rho}_{| G_v}) \subset \GL_2(\F)$ and $\F_{q_v^2} \subset \F$.  We will apply Raynaud's theory of vector space schemes \cite{Raynaud}.  An $\F$-vector space scheme over $D$ is a commutative group scheme $W / D$ carrying an action of $\F$.  Let $\mathcal{I} \subset \mathcal{O}_{W}$ be the augmentation ideal, so that $\mathcal{O}_W = \mathcal{I} \oplus \mathcal{O}_D$.  Here $\mathcal{O}_W$ is the structure sheaf of $W$.  It is easy to see that $\mathcal{I}$ decomposes as follows:
$$
\mathcal{I} = \bigoplus_{\chi: \F^\ast \to D^\ast} \mathcal{I}_{\chi},
$$
where $\mathcal{I}_{\chi}$ is the piece of $\mathcal{I}$ on which $\F$ acts via the character $\chi$.  We see that $H$ is an $\F_{q_v^2}$-vector space scheme, and it satisfies the additional crucial property that each $\mathcal{I}_{\chi}$ is a non-zero invertible sheaf.  The vector space scheme $H$ is endowed with two Galois actions:
\begin{enumerate}
\item As we noted before, $\Gal(\overline{K}/K) \simeq I_v$ acts on $H(\overline{K})$ by the character $\vphi$, which we are trying to determine.
\item $\Gal(K^\prime /K) \simeq k_v^\ast$ acts on the cotangent space $\mathrm{cot}(H_{D^\prime} \times_{D^\prime} \Fpbar)$.  Thanks to Carayol's congruences mentioned above, we can express this action explicitly in terms of the character $\theta$.
\end{enumerate}

From Raynaud's work one deduces an explicit relation between these two different Galois actions.  We will not perform the calculations here, but the reader can find them in Section 3 of \cite{ramifiedpaper}.  At the end we obtain a collection $\Phi(\theta)$ of characters $\vphi$ that are compatible with the known action of $\Gal(K^\prime /K)$.  It turns out that these are precisely the characters $\vphi$ arising from mod $p$ Galois representations $\overline{\rho}$ that are modular of some Serre weight $\sigma^\prime \tensor \sigma^v$, where $\sigma^\prime \in \mathrm{JH}(\mathrm{Ind}_{B(k_v)}^{\GL_2(k_v)} \theta)$ and $\sigma^v = \bigotimes_{w | p, w \neq v} \sigma_w$, where $\sigma_w$ is an arbitrary local Serre weight at $w$.  Observe that this is the best result that we can hope to obtain at this stage of the proof, since so far we have only used $\theta$ in our calculations and not $\sigma_v$ itself.

To get a more precise result, we consider all the characters $\theta: B(k_v) \to \Fpbar^\ast$ such that $\sigma_v \in \mathrm{JH}(\mathrm{Ind}_{B(k_v)}^{\GL_2(k_v)} \theta)$.  Clearly all the $\vphi$ associated to $\overline{\rho}$ that are modular of weight $\sigma_v \tensor \sigma^v$ lie in the intersection $\bigcap_{\sigma_v \in \mathrm{JH}(\mathrm{Ind} \theta)} \Phi_\theta$.  We hope that this intersection will turn out to be exactly the collection of representations $\overline{\rho}$ such that $\sigma_v \in W_v(\overline{\rho})$.  The hope comes true when $\sigma_v$ is of the form (\ref{lclsrwt}) with $0 \leq r_\tau \leq p - 1 - e_v$ for all $\tau \in S_v$, which is the reason for the hypothesis to this effect that we made above.   The combinatorial issues that prevent this method from giving us as good a theorem as we would like when $r_\tau$ does not satisfy the constraint $0 \leq r_\tau \leq p - 1 - e_v$ are essentially also what prevents the method of \cite{GeeBDJ} from handling the non-regular Serre weights.

\section{The mod $p$ local Langlands correspondence} \label{modpllc}
The Langlands philosophy postulates a deep connection between algebra and analysis and is one of the main motivations behind modern research on Serre's modularity conjecture and its generalizations.  In this section we will show a very brief glimmer of the connection between them.  Let $n \geq 1$, let $F / \Q_p$ be a $p$-adic field, and let $E$ be a field.  In very rough terms, we would like to have a correspondence between certain Galois representations $\rho: G_F \to \GL_n(E)$ and certain representations of $\GL_n(F)$ on vector spaces over $E$; one of the most difficult parts of this problem is finding the correct definition of ``certain.''  Often one can attach $L$-functions to each of these types of objects, and the $L$-functions of the objects paired by the correspondence should match.  

In the case of $E = \C$, the correspondence was proved by Harris and Taylor \cite{HT} and Henniart \cite{Henniart}, working with Weil-Deligne representations instead of the closely related Galois representations.
If $E = \overline{\F}_l$, with $l \neq p$, then considerable progress was made by Vign\'{e}ras \cite{Vigneras}.  However, if $E = \Fpbar$, then very little is known.  In many respects the study of the mod $p$ local Langlands correspondence is at the stage in its development where the complex local Langlands correspondence was in the 1970's: one tries to classify objects on both sides and pair them up explicitly in a natural way, but no deep underlying theory is yet available.  Moreover, understanding the mod $p$ representation theory of $\GL_n(F)$ has turned out to be remarkably difficult.

In this section, we will use the following notation.  We let $G = \GL_n(F)$ and consider the maximal open compact subgroup $K = \GL_n(\mathcal{O}_F)$ and the center $Z = Z(G) \simeq F^\ast$.  Let $\pi \in \mathcal{O}_F$ be a uniformizer, and let $k_F = \mathcal{O}_F / (\pi)$ be the residue field as usual.  Let $q = p^f$ be the cardinality of $k$.  Let $I \subset K$ be the Iwahori subgroup consisting of matrices that are upper triangular modulo $\pi$, and let $I(1)$ be the pro-$p$-Sylow subgroup of $I$.  For instance, if $n = 2$ then
\begin{eqnarray*}
I & = & \left\{ \left( \begin{array}{cc} a & b \\ c & d \end{array} \right) \in \GL_2(\mathcal{O}_F) : c \in \pi \mathcal{O}_F \right\} \\
I(1) & = & \left\{ \left( \begin{array}{cc} a & b \\ c & d \end{array} \right) \in \GL_2(\mathcal{O}_F) : c \in \pi \mathcal{O}_F; a,d \in 1 + \pi \mathcal{O}_F \right\}.
\end{eqnarray*}

Let $\sigma$ be an irreducible $\Fpbar$-representation of $K$.  By Proposition \ref{basicfact}, $\sigma$ factors through the natural reduction map $K \to \GL_n(k)$, since the kernel of this map is a pro-$p$ group.  Therefore, $\sigma$ arises from an irreducible $\Fpbar$-representation of $\GL_n(k)$ by inflation; these are exactly the objects that we called local Serre weights above in the case $n = 2$.  Moreover, we can view $\sigma$ as a representation of the larger group $KZ$ by decreeing that $\left( \begin{array}{cc} \pi & 0 \\ 0 & \pi \end{array} \right)$ acts trivially.

If $H \subset G$ is any open subgroup, and $\tau$ is an $\Fpbar$-representation of $H$, we can consider the compact induction $\ind_H^G \tau$.  A model for this representation is given by the space of functions $f : G \to V_\tau$ that are locally constant, compactly supported modulo $Z$, and satisfy the condition $f(hg) = \tau(h) \cdot f(g)$ for every $h \in H$ and $g \in G$.  Here $V_\tau$ is the underlying $\Fpbar$-vector space of $\tau$.  The action of $G$ is given by $(gf)(x) = f(xg)$ for all $g, x \in G$.  Note that if $H$ is a subgroup of finite index, then local constancy and compact support are automatic and this is just the usual induction.  The endomorphisms of this compact induction were computed by Barthel and Livn\'{e} \cite{BarthelLivne} for $n = 2$.  For $n \geq 2$, see \cite{Scheingln} for an explicit computation and \cite{FlorianSatake} for a more conceptual argument on the level of algebraic groups.

\begin{proposition} \label{heckestructure}
Let $\sigma$ be an irreducible $\Fpbar$-representation of $K$.  The endomorphism algebra $\mathrm{End}_{G}(\ind_{KZ}^G \sigma)$ is equal to a polynomial ring $\Fpbar[T_1, \dots, T_{n-1}]$, where the $T_i$ are explicitly defined endomorphisms.
\end{proposition}

Let $W$ be an irreducible $\Fpbar$-representation of $G$ with central character, i.e. such that the elements of $Z$ act by scalars.  Twisting by an unramified character, we may assume that $\left( \begin{array}{cc} \pi & 0 \\ 0 & \pi \end{array} \right)$ acts trivially.  If $\sigma \subset W_{| K}$ is a $K$-submodule of $G$, then by Frobenius reciprocity we obtain a non-zero homomorphism $\ind_{KZ}^G \sigma \to W$ of $G$-modules, which must be a surjection by the irreducibility of $W$.  We say that $W$ is {\emph{admissible}} if the space of invariants $W^U = \{ w \in W: \forall u \in U, uw = w \}$ is finite-dimensional for any open subgroup $U \subset G$; since $W$ is an $\Fpbar$-representation this is in fact equivalent to $W^{I(1)}$ being finite-dimensional.

The endomorphism algebra $\mathrm{End}_G(\ind_{KZ}^G \sigma)$ is commutative by Proposition \ref{heckestructure}, and it acts on $\mathrm{Hom}_G(\ind_{KZ}^G \sigma, W)$ in the obvious way.  If $W$ is assumed to be admissible, then $\mathrm{Hom}_G(\ind_{KZ}^G \sigma, W) \simeq \mathrm{Hom}_{KZ}(\sigma, W_{| KZ})$ is finite-dimensional (because $\sigma$ must contain a non-zero $I(1)$-invariant, which must map to an element of $W^{I(1)}$) and necessarily contains an eigenvector for the $\mathrm{End}_G(\ind_{KZ}^G \sigma)$-action.  We obtain the following result.

\begin{proposition}
Let $W$ be a smooth irreducible $\Fpbar[G]$-module with central character.  Assume that $W$ is admissible if $n \geq 3$.  Let $\sigma$ be an irreducible $\Fpbar[K]$-module $\sigma$ such that $\sigma \subset W_{| K}$.  Then there exist an unramified character $\chi: F^\ast \to \Fpbar^\ast$ and scalars $\lambda_1, \dots, \lambda_{n-1} \in \Fpbar$ such that there exists a surjection of $G$-modules
\begin{equation}
(\chi \circ \det) \tensor \ind_{KZ}^{G} \sigma / (T_1 - \lambda_1, \dots, T_{n-1} - \lambda_{n-1})\ind_{KZ}^G \sigma \twoheadrightarrow W.
\end{equation}
\end{proposition}
\begin{proof}
If $W$ is admissible, then we have sketched out the proof.  If $n = 2$, then Barthel and Livn\'{e} (see Theorems 32 and 33 of \cite{BarthelLivne}) obtain this result without assuming admissibility of $W$ by using the fact that $\mathrm{End}_G(\ind_{KZ}^G \sigma)$ has Krull dimension 1.
\end{proof}

For the rest of this section, suppose that $n = 2$.  In this case, the endomorphism algebra $\mathrm{End}_G(\ind_{KZ}^G \sigma)$ has a single generator $T_1$, which we will call $T$.  Up to unramified twist, we know that every irreducible $\Fpbar[G]$-module with central character is a quotient of $\ind_{KZ}^G \sigma / (T - \lambda)(\ind_{KZ}^G \sigma)$ for some $\sigma$ and some $\lambda \in \Fpbar$.  We say that an irreducible $W$ as above is {\emph{supersingular}} if it is a quotient of some $\ind_{KZ}^G \sigma / (T - \lambda)(\ind_{KZ}^G \sigma)$.  Barthel and Livn\'{e} proved a partial classification of the irreducible $\Fpbar[G]$-modules with central character as follows.  Note that if $G = \GL_2(\Q_p)$, then Berger \cite{Berger} recently showed that all irreducible $\Fpbar[G]$-modules have central character, but this is not known even for $G = \GL_2(F)$ whenever $F \neq \Q_p$.

\begin{theorem}[Barthel-Livn\'{e}]
Let $\sigma$ be an irreducible $\Fpbar[K]$-module.
\begin{enumerate}
\item If $\sigma$ has dimension other than $1$ or $p^f$ (the minimal and maximal dimensions possible) or if $\lambda \neq \pm 1$, then $\ind_{KZ}^G \sigma / (T - \lambda)(\ind_{KZ}^G \sigma)$ is irreducible and is isomorphic to the parabolic induction of a character from the upper triangular Borel subgroup $B \subset G$.
\item The induction $\ind_B^G \mathbf{1}$, where $\mathbf{1}$ is the trivial character of $B$, has length two.  Its subquotients are a one-dimensional representation $\det$ and an infinite-dimensional analogue of the Steinberg representation, denoted $\mathrm{St}$.
\item Up to unramified twist, every smooth irreducible $\Fpbar[G]$-module $W$ with central character satisfies exactly one of the following statements:
\begin{enumerate}
\item $W \simeq \ind_{KZ}^G \sigma / (T - \lambda)(\ind_{KZ}^G \sigma)$, where $\sigma$ has dimension other than $1$ or $p^f$, or $\lambda \neq \pm 1$.
\item $W \simeq \chi \circ \det$ for some smooth character $\chi: F^\ast \to \Fpbar^\ast$.
\item $W \simeq (\chi \circ \det) \tensor \mathrm{St}$ for some smooth character $\chi: F^\ast \to \Fpbar^\ast$.
\item $W$ is supersingular.
\end{enumerate}
\end{enumerate}
\end{theorem}

\begin{remark}
The previous theorem classifies all non-supersingular (smooth, with central character) $\Fpbar$-representations of $\GL_2(F)$ for arbitrary finite extensions $F / \Q_p$.  Herzig \cite{HerzigParabInd} proved a generalization of this theorem of $\GL_n(F)$ for $n > 2$, in which all smooth admissible representations of $\GL_n(F)$ with central character are classified in terms of the supersingular representations of $\GL_m(F)$ for $m \leq n$.  A representation of $\GL_n(F)$ is called supersingular if it is a quotient of $\ind_{KZ}^G \sigma / (T_1, \dots, T_{n-1})$.  Abe \cite{Abe} further generalized this result to a wider class of reductive groups.
\end{remark}

Let $L$ be a number field and $v$ a place of $L$ such that $L_v \simeq F$.  If $\rho: G_F \to \GL_2(\Fpbar)$ is an irreducible local Galois representation, let $\tilde{\rho}: G_L \to \GL_2(\Fpbar)$ be a global representation such that $\tilde{\rho}_{| G_v} \simeq \rho$.  Recall that in Conjecture \ref{strongserre} we defined a set $W_v(\tilde{\rho})$ of local Serre weights at $v$, which in fact depends only on $\rho$.  Thus we can speak of a set $W(\rho)$ of modular local Serre weights.

Now suppose that $G = \GL_2(\Q_p)$.  In this case, the irreducible $\Fpbar[K]$-modules have the form $\sigma = \det^w \tensor \Sym^r \Fpbar^2$ with $0 \leq w \leq p -2$ and $0 \leq r \leq p - 1$.  We define an involution on the set of these local Serre weights as follows.  For $\sigma$ as above, define $\sigma^\prime = \det^{w + r} \tensor \Sym^{p - 1 - r} \Fpbar^2$.  Note that $(\sigma^\prime)^\prime = \sigma$.  It is easy to compute from the statement of Conjecture \ref{strongserre} that if $\rho: G_{\Q_p} \to \GL_2(\Fpbar)$ is irreducible, then $W(\rho)$ is necessarily of the form $W(\rho) = \{ \sigma, \sigma^\prime \}$.

Breuil \cite{Breuil} completed the classification of the irreducible $\Fpbar$-representations of $\GL_2(\Q_p)$ with the following result.

\begin{theorem} [Breuil]
Let $G = \GL_2(\Q_p)$.  Then for every local Serre weight $\sigma$, the $G$-module $\ind_{KZ}^G \sigma / T(\ind_{KZ}^G \sigma)$ is irreducible.  Moreover, for every $\sigma$ we have
\begin{equation} \label{intertwine}
\ind_{KZ}^G \sigma / T(\ind_{KZ}^G \sigma) \simeq \ind_{KZ}^G \sigma^\prime / T(\ind_{KZ}^G \sigma^\prime)
\end{equation}
and these are the only isomorphisms among supersingular $\Fpbar[G]$-modules.
\end{theorem}
\begin{remark}
Note that the two operators $T$ appearing in (\ref{intertwine}) are different objects.  The $T$ on the left-hand side is the generator of the endomorphism algebra of $\ind_{KZ}^G \sigma$, while the one on the right-hand side generates the endomorphism algebra of $\ind_{KZ}^G \sigma^\prime$.
\end{remark}
\begin{proof}
Let $W = \ind_{KZ}^G \sigma / T(\ind_{KZ}^G \sigma)$ and let $U \subset W$ be an irreducible $G$-submodule.  By explicit computation, one shows that $W^{I(1)}$ is two-dimensional and that every non-zero element of $W^{I(1)}$ generates $W$ as a $G$-module.  But $U^{I(1)} \neq 0$ by Proposition \ref{basicfact} and hence $U = W$.  The isomorphisms of (\ref{intertwine}) are constructed explicitly, and one shows that
\begin{equation} \label{socle}
 \mathrm{soc}_K (\ind_{KZ}^G \sigma / T(\ind_{KZ}^G \sigma)) \simeq \sigma \oplus \sigma^\prime,
 \end{equation}
implying that there are no other isomorphisms.  Recall that for a $G$-module $M$, the socle $\mathrm{soc}_K (M)$ is the direct sum of all irreducible $K$-submodules of $M$.
\end{proof}

If $\rho: G_{\Q_p} \to \GL_2(\Fpbar)$ is an irreducible local Galois representation, define an $\Fpbar$-representation of $\GL_2(\Q_p)$ by $\pi(\rho) = \ind_{KZ}^G \sigma / T(\ind_{KZ}^G \sigma)$, where $\sigma \in W(\rho)$.  It is immediate from the results just presented that $\pi(\rho)$ is well-defined and that this construction provides a bijection between irreducible Galois representations $\rho: G_{\Q_p} \to \GL_2(\Fpbar)$ and supersingular representations of $\GL_2(\Q_p)$.  Note that the following relation is satisfied:
\begin{equation} \label{soclerel}
\mathrm{soc}_K (\pi(\rho)) = \bigoplus_{\sigma \in W(\rho)} \sigma.
\end{equation}

In the same paper \cite{Breuil}, Breuil constructed $\pi(\rho)$ for semisimple reducible $\rho$, and eventually Colmez defined $\pi(\rho)$ for indecomposable $\rho$, thereby completing the mod $p$ local Langlands correspondence for $\GL_2(\Q_p)$.  These constructions are more complicated than the one presented above, and we will not give them here, nor shall we argue why these definitions of $\pi(\rho)$ are the ``correct'' ones.  However, it is important to note that the property (\ref{soclerel}) remains true for all $\rho$.

If $F \neq \Q_p$, then almost nothing is known about the mod $p$ local Langlands correspondence for $\GL_2(F)$, and the statements that are known are almost all negative.  For instance, we know that  there cannot be a bijection between irreducible Galois representations $\rho: G_F \to \GL_2(\Fpbar)$ and supersingular representations of $\GL_2(F)$ because there are far too many of the latter.  The condition (\ref{soclerel}) does not isolate a $\pi(\rho)$ because, for unramified extensions $F / \Q_p$, Breuil and Paskunas \cite{BreuilPaskunas} have proved the existence of infinite families of supersingular representations $W$ satisfying $\mathrm{soc}_K(W) \simeq \bigoplus_{\sigma \in W(\rho)} \sigma$.  Moreover, although it is immediate from Zorn's Lemma that supersingular representations of $\GL_2(F)$ exist, we do not have a single explicit construction of one; the proof of Breuil and Paskunas uses the theory of diagrams and involves the taking of injective envelopes, which makes their work very non-explicit.  In fact, Schraen \cite{Schraen} has shown that if $F / \Q_p$ is quadratic, then no supersingular representation of $\GL_2(F)$ is finitely presented.  This makes it difficult to contemplate generalizations of Colmez's construction.

Let $e$ be the ramification index of $F / \Q_p$, and recall that the residue field $k$ of $F$ satisfies $[k : \F_p] = f$.  Let $F_0$ be the maximal unramified subextension of $F / \Q_p$, and observe that, since $F$ and $F_0$ have the same residue field, the Serre weights for $F_0$ are the same as those for $F$.  If $\rho: G_F \to \GL_2(\Fpbar)$ is irreducible, then (see \cite{cuspidalpaper}) one defines $e^f$ irreducible representations $\rho_1, \dots, \rho_{e^f} : G_{F_0} \to \GL_2(\Fpbar)$ such that one expects
$$ \mathrm{soc}_K (\pi(\rho)) \simeq \bigoplus_{i = 1}^{e^f} \bigoplus_{\sigma \in W(\rho_i)} \sigma.$$

The formula above explains the multiplicities of the different constituents of the $K$-socle of $\pi(\rho)$.  Ongoing work of Breuil and Diamond aims to specify the $K$-socles of $\pi(\rho)$ for reducible $\rho$.  This section has only scratched the surface of the mod $p$ local Langlands correspondence and has said almost nothing about current research, but we hope that it has sufficiently piqued the reader's interest to consult the literature for more details about the field.

\section{Potential modularity and compatible systems}
After the digression about mod $p$ local Langlands in the previous section, we return to our discussion of Serre's modularity conjecture.  In particular, we return to the notation of Section \ref{statement}, so that $F$ is now again a totally real number field.

\subsection{A wish list} \label{wishlist}
Suppose that we have two mod $p$ Galois representations $\overline{\rho}_1: G_F \to \GL_2(\overline{\F}_{p_1})$ and $\overline{\rho}_2: G_F \to \GL_2(\overline{\F}_{p_2})$, where $p_1$ and $p_2$ are two primes, possibly distinct.  It clearly would be useful to be able to prove statements of the form ``if $\overline{\rho}_1$ is modular and certain conditions are satisfied, then $\overline{\rho}_2$ is modular as well.''  Such theorems would allow us to leverage knowledge of Serre's conjecture in some special cases to prove it for larger classes of Galois representations.

How can we relate the modularity of two different Galois representations?  A crucial idea is to think about $p$-adic Galois representations, and to recall that when we first encountered them, in Prof. B\"{o}ckle's lectures, they were constructed in families.  Indeed, for a modular form $f$, we obtained a representation $\rho_{f, l} : G_\Q \to \GL_2(\overline{\Q}_l)$ for each prime $l$.  The $\rho_{f, l}$ for different $l$ were very intimately related.  

An important starting point for work on modularity is an axiomatization of this phenomenon: the notion of weakly and strongly compatible systems that we saw in Prof. B\"{o}ckle's lectures.  A strongly compatible system $\{ \rho_l : G_\Q \to \GL_2(\overline{\Q}_l) \}$ of Galois representations behaves like a family of representations arising from a modular form.  In particular, if one member of the system is modular, then they all are, and the same is true of their reductions.  This gives us a general strategy for proving the ``if $\overline{\rho}_1$ is modular, then $\overline{\rho}_2$ is too'' theorems that we wished for at the beginning of this section.  Suppose we could find a compatible system $\{ \rho_l \}$ such that $\overline{\rho}_1 \simeq \overline{\rho_{p_1}}$ and $\overline{\rho}_2 \simeq \overline{\rho_{p_2}}$.  We are assuming that $\overline{\rho}_1$ is modular.  If we could somehow prove that $\rho_{p_1}$ is modular, the compatible system would allow us to conclude that $\rho_{p_2}$ is modular as well, and hence that $\overline{\rho}_2$ is modular.

Three major ingredients are involved in implementing this strategy.  Starting with a representation $\overline{\rho} : G_\Q \to \GL_2(\Fpbar)$, we have the following {\bf{wish list}}:

\begin{enumerate}
\item Find nice lifts $\rho : G_\Q \to \GL_2(\overline{\Q}_p)$ of $\overline{\rho}$.  ``Nice'' will mean that $\rho$ satisfies hypotheses that make the other parts of the wish list available.
\item Given such a lift $\rho$, embed it in a compatible system $\{ \rho_l \}$ such that $\rho \simeq \rho_p$.
\item Modularity lifting theorems.
\end{enumerate}

\subsection{Potential modularity of mod $p$ Galois representations}
In this section we will sketch a proof of the following theorem of Taylor \cite{TaylorJussieu}.

\begin{proposition} \label{modppotmod}
Let $\overline{\rho} : G_\Q \to \GL_2(k)$ be a continuous, irreducible, odd Galois representation, where $k$ is a finite field of characteristic $p$.  Then there exists a Galois totally real extension $F / \Q$ that is unramified at $p$ and such that $\overline{\rho}_{| G_F}$ is modular.
\end{proposition}

If $\overline{\rho}$ has solvable image, then this problem may be handled by the methods of Langlands and Tunnell, so we will assume that this is not the case.  Passing to a suitable totally real extension $F / \Q$ that is unramified at $p$, we may assume that the determinant of $\overline{\rho}_{| G_v}$ is the mod $p$ cyclotomic character for all places $v | p$ of $F$, and that $\overline{\rho}$ has the following form at all $v | p$: 
$$ \overline{\rho}_{| G_v} \sim \left( \begin{array}{cc} \varepsilon \chi_v^{-1} & \ast \\ 0 & \chi_v \end{array} \right),$$
where $\varepsilon$ is the mod $p$ cyclotomic character and $\chi_v : G_v \to k^\ast$ is a character. Now $\overline{\rho}_{G_F}$ looks like it could be the restriction of an ordinary $p$-adic representation of Hodge-Tate weights $\{ 0, 1 \}$ coming from an abelian variety, and our task is to show that this is indeed the case.

Recall that if $A / F$ is an abelian variety, then for every finite place $v$ of $F$, the Galois group $G_F$ acts on the torsion $A[v]$, and the reductions of these representations give us a strictly compatible system $\{ \overline{\rho}_{A, v} \}$.  We are looking for an abelian variety $A / F$ such that $\overline{\rho}_{| G_F} \simeq \overline{\rho}_{A, v}$ for some $v$.  Let $v^\prime \neq v$ be another place of $F$ lying over $p$.  We will cleverly set up a moduli problem of abelian varieties in such a way that a handy theorem of Morel-Bailly \cite{MB}, quoted below, will give us exactly the existence of the $A$ that we need.

Let $M / \Q$ be an imaginary quadratic extension and $\psi: G_M \to \overline{\Q}_p^\ast$ a character.  Consider the moduli problem of triples $(A, \vphi_v, \vphi_{v^\prime})$ such that $A$ is a Hilbert-Blumenthal abelian variety (this is an abelian variety carrying an action of the ring of integers $\mathcal{O}_E$ of a specified totally real field $E$ and some additional structure such as a Rosati involution and a polarization; we will not give a precise definition here but only mention that there is a well-developed theory of moduli problems for these objects, founded by Rapoport in \cite{Rapoport}), and the $\vphi_v$ and $\vphi_{v^\prime}$ are isomorphisms $\vphi_v: \overline{\rho}_{| G_F} \stackrel{\sim}{\to} \overline{\rho}_{A, v}$ and $\vphi_{v^\prime} : \mathrm{Ind}_{G_M}^{G_\Q} \overline{\psi} \stackrel{\sim}{\to} \overline{\rho}_{A, v^\prime}$.  The general theory of Hilbert-Blumenthal abelian varieties tells us that this moduli problem is representable by some moduli space $X / \Q$, and if we knew that this space had a rational point, it would correspond to the abelian variety $A$ that we are looking for.

\begin{proposition}[Moret-Bailly]
Let $K$ be a number field and $S$ a finite set of places of $K$.  If $X / K$ is a geometrically irreducible smooth quasi-projective scheme and $X(K_v) \neq \varnothing$ for all $v \in S$, then $X(K_S)$ is Zariski dense in X.  Here $K_S / K$ is the maximal extension of $K$ in which all $v \in S$ split completely.
\end{proposition}

By choosing $E, v, v^\prime, \psi$ wisely, it can be arranged that the hypotheses of Moret-Bailly's theorem are satisfied for $K = \Q$ and $X / \Q$ the moduli space considered above.  In fact, Moret-Bailly's result appears to be far stronger than what we need to prove the existence of a rational point.  This gives us the freedom to strengthen Proposition \ref{modppotmod} by imposing a number of additional properties on the totally real field $F$, such as requiring it to be linearly disjoint from any specified number field. These strengthenings turn out to be essential, as they allow $\overline{\rho}_{| G_F}$ to satisfy the hypotheses of the modularity lifting theorems that we will call upon later.

\subsection{Deformation theory and modularity lifting results} \label{lifting}
In the previous section we laid out the ingredients of the proof of a potential modularity theorem for mod $p$ representations.  Now we want to build on that result to get a potential modularity theorem for $p$-adic representations, which will be used in Section \ref{compsys}.  Suppose that we are given a continuous, odd, irreducible mod $p$ Galois representation $\overline{\rho}: G_\Q \to \GL_2(k)$, as usual.  First of all, we want to find a nice $p$-adic lift of $\overline{\rho}$ as in the first item of the wish list of Section \ref{wishlist}.  

Consider the following deformation problem.  We want to study deformations $\rho: G_\Q \to \GL_2(A)$, where $A$ is a complete local noetherian algebra with residue field $k$, such that $\rho$ lifts $\overline{\rho}$.  In addition, 
for each prime $l$ we fix an equivalence class $\tau_l$ of representations of the inertia group $I_l \simeq G_{\Q_l^{nr}}$, such that all but finitely many of the $\tau_l$ are trivial.  Let $\chi_p: G_\Q \to \overline{\Q}_p^\ast$ be a character; we have $\chi_p = \omega_p^{k-1}$, where $\omega_p$ is the $p$-adic cyclotomic character.  We require that $\det \rho = \chi_p$, that $\tau_l$ be the restriction to inertia of the Weil-Deligne representation associated to $\rho_{| G_l}$ for each $l$, and that $\rho_{| G_p}$ be crystalline of Hodge-Tate weights $\{ 0, k-1 \}$.

By general deformation theory, this deformation problem is represented by a complete noetherian local ring $R^X_{\overline{\rho}, \Q}$ with residue field $k$.  It can be proved with very considerable effort, using Galois cohomology and the Euler characteristic formula (see \cite{Boeckle}), that $\dim R^X_{\overline{\rho}, \Q} \geq 1$ (by the dimension of a ring we mean the Krull dimension).  To get our nice lift of $\overline{\rho}$, we need to show that $R^X_{\overline{\rho}, \Q}$ has a point over an algebra of characteristic zero.  

Consider the Hecke algebra $\T = \Z_p[T_l: l \neq p]$, which acts on the space of modular forms $S_{k}(\Gamma_1(N(\overline{\rho}))$.  We get a natural surjection 
\begin{equation} \label{rtsurj}
R^X_{\overline{\rho}, \Q} \twoheadrightarrow \T_{\mathfrak{m}_{\overline{\rho}}}
\end{equation}
by the universal property of $R^X_{\overline{\rho},\Q}$.  On the other hand, if every deformation classified by $R^X_{\overline{\rho},\Q}$ is modular, then the universal deformation must factor through $\T_{\mathfrak{m}_{\overline{\rho}}}$ and therefore $R^X_{\overline{\rho},\Q} \simeq \T_{\mathfrak{m}_{\overline{\rho}}}$.  Proving a modularity lifting theorem, therefore, comes down to proving such an isomorphism, i.e. an ``$R = T$'' theorem.  One of the breakthroughs of the Taylor-Wiles method \cite{Wiles} \cite{TaylorWiles} was the understanding that modularity lifting results could often be reduced to statements about ring-theoretic properties of Hecke algebras.  For example, if $R^X_{\overline{\rho}, \Q}$ were an integral domain, then proving $\dim R^X_{\overline{\rho}, \Q} = \dim \T_{\mathfrak{m}_{\overline{\rho}}}$ would suffice to establish that (\ref{rtsurj}) is an isomorphism, since the quotient of an integral domain by a non-trivial ideal has strictly lower Krull dimension than the original ring.  In general, $\Spec R^X_{\overline{\rho}, \Q}$ will have more than one irreducible component, and proving that (\ref{rtsurj}) is an isomorphism often amounts to showing that each component overlaps the image of $\Spec \T_{\mathfrak{m}_{\overline{\rho}}}$ as well as comparing Krull dimensions.

In fact, we do not know that (\ref{rtsurj}) is an isomorphism.  However, we know from Proposition \ref{modppotmod} that $\overline{\rho}_{| G_F}$ is modular for some totally real fields $F$.  We may consider an analogous deformation problem to the one studied above, but over $F$; it is represented by a deformation ring $R^X_{\overline{\rho}, F}$.  Moreover, if we choose $F$ correctly, then a modularity lifting theorem is known by work of Diamond \cite{Diamond} and Fujiwara \cite{Fujiwara}; in that case one can prove that $R^X_{\overline{\rho}, F}$ is isomorphic to a suitable localized Hecke algebra $\T_F$.

The map $\Spec R^X_{\overline{\rho}, \Q} \to \Spec R^X_{\overline{\rho}, F}$ corresponding to restriction to the subgroup $G_F$ of representations of $G_\Q$ is clearly quasi-finite, i.e. has finite fibers.  Indeed, $G_F$ has finite index in $G_\Q$ and it is not hard to see that there are only finitely many ways to extend a representation of $G_F$ to the larger group $G_\Q$.  Moreover, the Hecke algebra $\T_F$ is finitely generated as a $\Z_p$-module, since it embeds in the endomorphism algebra of a suitable abelian variety.  This implies that $\T_F / (p)$ is a finite set, therefore that $R^X_{\overline{\rho}, F} / (p)$ is finite, and therefore that $R^X_{\overline{\rho}, \Q} / (p)$ is finite, hence has dimension zero.  Hence, any prime ideal of $R^X_{\overline{\rho}, \Q}$ containing $(p)$ is necessarily maximal.

On the other hand, recall that $\dim R^X_{\overline{\rho}, \Q} \geq 1$.  This means that there exists a non-maximal prime ideal $P \subset R^X_{\overline{\rho}, \Q}$.  By the above, we know that $P$ does not contain $(p)$.  Since $R^X_{\overline{\rho}, \Q}$ is finitely generated over $\Z_p$, it follows that the quotient $R^X_{\overline{\rho}, \Q} / P$ embeds into the ring of integers $\mathcal{O}_L$ of a suitable finite extension $L / \Q_p$.  Now by the universal property of $R^X_{\overline{\rho}, \Q}$, the embedding $R^X_{\overline{\rho}, \Q} / P \hookrightarrow \mathcal{O}_L$ corresponds to a $p$-adic Galois representation $\rho: G_\Q \to \GL_2(\mathcal{O}_L)$ lifting $\overline{\rho}$.

We have now achieved the first part of the wish list in Section \ref{wishlist}.  In fact, by doing all of this more carefully we could ensure that the obtained lift $\rho$ has a variety of good properties.

\subsection{Constructing compatible systems} \label{compsys}
In the previous section, we started with a mod $p$ representation $\overline{\rho}: G_\Q \to \GL_2(\Fpbar)$ and found a finite extension $L / \Q_p$ and a $p$-adic representation $\rho: G_\Q \to \GL_2(\mathcal{O}_L)$ lifting $\overline{\rho}$, thereby fulfilling the first part of the wish list of section \ref{wishlist}.  In this section we will build a compatible system around $\rho$.

By Taylor's potential modularity theorem (Proposition \ref{modppotmod}), we know that there is a Galois totally real field $F / \Q$ such that $\overline{\rho}_{| G_F}$ is modular.  By the modularity lifting theorems of Diamond and Fujiwara that were mentioned in the previous section, we know that $\rho_{| G_F}$ is modular as well.  Let $G = \Gal(F / \Q)$.  By Brauer's theorem (see, for instance, chapter 10 of \cite{Serre}) there exist solvable subgroups $H_i \subset G$, integers $n_i \in \Z$, and one-dimensional representations $\chi_i$ of $H_i$ such that
\begin{equation}
\mathbf{1} = \sum_{i = 1}^t n_i \mathrm{Ind}_{H_i}^G \chi_i,
\end{equation}
in the Grothendieck group of $G$, where $\mathbf{1}$ is the trivial representation of $G$.  Note that even though $\mathbf{1}$ is a true representation, some of the $n_i$ might be negative.  This will cause us problems later.  Set $F_i$ to be the fixed field of $H_i$.  Tensoring with $\rho$, we obtain that
\begin{equation} \label{gfieq}
\rho = \sum_{i = 1}^t n_i \mathrm{Ind}_{G_{F_i}}^{G_\Q} (\rho_{| {G_{F_i}}} \tensor \chi_i).
\end{equation} 

Since $\Gal(F / F_i) = H_i$ is solvable, we conclude by Langlands-Tunnell solvable base change that each $\rho_{| G_{F_i}}$ arises from an automorphic form $\pi_i$ on $F_i$.  Hence we can trivially rewrite (\ref{gfieq}) as
\begin{equation}
\rho = \rho_p = \sum_{i = 1}^t n_i \mathrm{Ind}_{G_{F_i}}^{G_\Q} (\rho_{\pi_i, p} \tensor \chi_i).
\end{equation}

Since each $\rho_{\pi_i, p}$ comes from an automorphic form and therefore sits in a compatible system of representations of $G_{F_i}$, it is very tempting to define
\begin{equation}
\rho_l = \sum_{i = 1}^t n_i \mathrm{Ind}_{G_{F_i}}^{G_\Q} (\rho_{\pi_i, l} \tensor \chi_i)
\end{equation}
for arbitrary primes $l$.  In fact, this idea works.  If we knew that the $\rho_l$ were true representations and not just virtual ones, then the compatible system properties of the $\{ \rho_{\pi_i, l} \}$ would easily imply that $\{ \rho_l \}$ is a compatible system as well.  In fact, it can indeed be checked that the $\{ \rho_l \}$ are true representations.  This fulfills the second part of the wish list.

\section{Proof of Serre's conjecture}
We are finally in a position to give an exceedingly impressionistic sketch of the strategy behind the proof of Serre's conjecture.  For more detail, the reader is referred to Wintenberger's excellent expository article \cite{WintenbergerBourbaki} and to Khare's exposition \cite{KhareSurvey}, which has somewhat fewer details but paints the big picture in bold strokes.  For simplicity, we will only consider the level one case of Serre's conjecture.  This means that we start with a Galois representation $\overline{\rho} : G_\Q \to \GL_2(\Fpbar)$ that is continuous, irreducible, odd, and unramified outside $p$.  Recall from Section \ref{levelsection} that the lack of ramification outside $p$ means that the prime-to-$p$ part of the Artin conductor $\mathfrak{n}(\overline{\rho})$ is trivial, and hence $N(\overline{\rho}) = 1$.  We aim to prove that $\overline{\rho}$ is modular.

It is important to note that some special cases of Serre's conjecture were known well before Khare's idea of applying Taylor's potential modularity results and Kisin's modularity lifting techniques to this problem.  In the 1970's Tate used discriminant bounds to prove that there are no continuous irreducible odd representations $\overline{\rho} : G_\Q \to \GL_2(\overline{\F}_2)$, and therefore that the level one case of Serre's modularity conjecture is vacuously true for $p = 2$.  Serre extended his argument to $p = 3$ shortly afterwards, and these two results are essential to the work of Khare and Wintenberger, since they constitute the base cases of their induction argument.  We note that the level one case of Serre's conjecture for $p = 5$ was proved by Brueggeman \cite{Brueggeman} contingent on the generalized Riemann hypothesis, and that, with some local hypotheses at small primes but no restriction on the level, the conjecture was proved for $\overline{\rho}$ with image lying in $\GL_2(\F_7)$ by Manoharmayum \cite{Manoharmayum} and for $\overline{\rho}$ with image lying in $\GL_2(\F_9)$ by Ellenberg \cite{Ellenberg}.

To give a taste of the inductive argument that proves Serre's conjecture, and to illustrate its crucial reliance on modularity lifting theorems, we will first flagrantly disregard the current reality and describe what the proof would look like if modularity lifting technology were more advanced than it actually is.  Assume the following, for the moment:

\begin{dream}
Let $\rho: G_\Q \to \GL_2(\overline{\Q}_p)$ be continuous, irreducible, unramified outside $p$, and crystalline at $p$ with Hodge-Tate weights $\{ w, 0 \}$ for some $w \leq 2p$.  Suppose that its reduction $\overline{\rho}$ is modular.  Then $\rho$ is modular.
\end{dream}

This dream follows, of course, from the Fontaine-Mazur conjecture.  It was considered totally out of reach when Khare and Wintenberger did their work, but such a modularity lifting result has since been proved in most cases by Kisin \cite{KisinFM}.  In fact, since Serre's conjecture is now known, his work implies most cases of Fontaine-Mazur for two-dimensional representations of $G_\Q$.  In any case, let us assume the dream and then prove Serre's conjecture in level one.

Let $p_n$ be the $n$-th prime.  We will prove Serre's conjecture by induction on $n$.  It is known for $p_1 = 2$ and $p_2 = 3$ by the theorems of Tate and Serre that were mentioned above.  Suppose that it is also known for $p_{n-1}$.  Let $\overline{\rho}: G_\Q \to \GL_2(\overline{\F}_{p_{n}})$ be continuous, irreducible, odd, and unramified outside $p_n$.  By the methods of Sections \ref{lifting} and \ref{compsys}, we can find a lift $\rho$ of $\overline{\rho}$ that sits in a compatible system $\{ \rho_l \}$, so that $\rho_{p_n} \simeq \rho$.  Consider $\overline{\rho}_{p_{n-1}}$; it is modular by induction.  Moreover, by the properties of compatible systems, $\rho_{p_{n-1}}$ is unramified outside $p_{n-1}$ and is crystalline of Hodge-Tate weight $(0, k(\overline{\rho}) - 1)$.  As we saw at the beginning of these notes, up to a twist we can assume that $k(\overline{\rho}) \leq p_n \leq 2 p_{n-1}$, where the second inequality is Bertrand's Postulate.  By the Dream, $\rho_{p_{n-1}}$ is modular.  Hence $\rho_{p_n}$ is modular by the compatible system, and hence $\overline{\rho}$ is modular and we are done.

The powerful modularity lifting theorem of the Dream can be seen as a fulfillment of the third part of the wish list of Section \ref{wishlist}.  Even though the Dream is not yet known, the modularity lifting theorems available to Khare and Wintenberger in 2005 were enough to prove Serre's conjecture, albeit with lots of technical work.  The modularity lifting theorems available now, and still more those available then, come with long lists of technical hypotheses, and one must be very careful to ensure that the liftings of $\overline{\rho}$ and the compatible systems obtained from the methods of Sections \ref{lifting} and \ref{compsys} satisfy these.  In these notes we have entirely ignored these technical points, which complicate the work tremendously.  However, at its core the basic idea is the simple one presented here.

We conclude with the unfortunate observation that it does not appear to be possible, at least not without a major new idea, to generalize the beautiful argument of Khare and Wintenberger to obtain a proof of the generalizations of Serre's conjecture to totally real fields that were incorporated into Conjecture \ref{strongserre} above.  While all the ingredients of their proof -- potential modularity, construction of lifts, compatible systems, modularity lifting theorems -- are less developed for arbitrary totally real fields than for $\Q$, a more fundamental problem is that any inductive argument on the places of $F$ would require that enough base cases be proved first, and we have no idea how to prove them.  Recall that Tate and Serre proved (the level one case of) Serre's conjecture for $p = 2$ and $p = 3$ by showing that it was vacuously true, i.e. that there were no $\overline{\rho}$ that were continuous, irreducible, odd, and unramified outside $p$.  While analogous non-existence theorems have been proved for some small primes and a few specific quadratic real fields (see, for instance, \cite{MT} and \cite{Sengun}), we know that for general totally real fields, even quadratic ones, Serre's conjecture is never vacuously true.  Indeed, for general totally real fields, for all $p$ there exist continuous, irreducible, odd mod $p$ Galois representations that are unramified outside $p$; see the introduction to \cite{BDJ} for an example in the case of $F = \Q(\sqrt{29})$.  Thus, to get the base case for an induction argument, one would need to establish a sufficiently large number of non-vacuous cases of Serre's conjecture, and it is not clear at all at the present time how to attack this problem.  Serre's modularity conjecture will likely continue to be an important motivation and source of research problems for some time to come.

\bibliographystyle{plain}
\bibliography{smclux}

\end{document}